\documentclass[10pt]{amsart}
\usepackage[pdftex,colorlinks]{hyperref}
\usepackage{amsmath}
\usepackage{epsfig}
\usepackage{amsthm}
\usepackage{amssymb}
\usepackage{euscript}

\DeclareFontFamily{U}{mathx}{\hyphenchar\font45}
\DeclareFontShape{U}{mathx}{m}{n}{
      <5> <6> <7> <8> <9> <10>
      <10.95> <12> <14.4> <17.28> <20.74> <24.88>
      mathx10
      }{}
\DeclareSymbolFont{mathx}{U}{mathx}{m}{n}
\DeclareMathAccent{\widecheck}    {0}{mathx}{"71}
\numberwithin{equation}{section}
\theoremstyle{plain}
\newtheorem{definition}{Definition}[section]
\newtheorem{theorem}{Theorem}[section]
\newtheorem{lemma}{Lemma}[section]

\newtheorem{remark}{Remark} [section]

\usepackage{amsmath, amsthm, amssymb, latexsym, mathrsfs}

\begin{document}
\title[Lioville theorem for high order degenerate elliptic equations]{A Liouville theorem for high order degenerate elliptic equations}
\author[Huang]{Genggeng Huang}
\address{Department of Mathematics\\
INS and MOE-LSC\\
Shanghai Jiao Tong University, Shanghai} \email{genggenghuang@sjtu.edu.cn}
%\address{Beijing International Center for Mathematical Research\\
%Peking University\\
%Beijing, 100871, China} \email{qhan@math.pku.edu.cn}
\author[Li]{Congming Li}
\address{Department of Mathematics\\
INS and MOE-LSC\\
Shanghai Jiao Tong University, Shanghai}\address{Department of Applied Mathematics, University of Colorado at Boulder}
\email{Congming.Li@Colorado.EDU}
\date{}
\maketitle

\begin{abstract}
In this paper, we  apply the moving plane method to the following high order degenerate elliptic equation,\begin{equation*}
(-A)^p u=u^\alpha\text{ in } \mathbb R^{n+1}_+,n\geq 1,
\end{equation*}where the operator $A=y\partial_y^2+a\partial_y+\Delta_x,a\geq 1$.  We get a Liouville  theorem for subcritical case and classify the solutions for the critical case.


\end{abstract}

\section{Introduction}\setcounter{section}{1} \setcounter{equation}{0}
\setcounter{theorem}{0}\setcounter{lemma}{0}
\label{intro}
This article concerns the symmetry of solutions of  degenerate elliptic equations on an unbounded domain. The first well-known work was first done by Gidas, Ni and Nirenberg \cite{GidasNiNirenberg} for the uniformly elliptic equations. In the elegant paper of \cite{GidasNiNirenberg}, one of the interesting results is on the symmetries of the non-negative solutions of \begin{equation}
\label{intro1}\Delta u + u^\alpha=0,\quad x\in \mathbb R^n, n\geq 3.
\end{equation}
They classified the positive solutions of \eqref{intro1} for $\alpha=\frac {n+2}{n-2}$ with additional decay at infinity, namely $u(x)=O(|x|^{2-n})$ by the method of moving plane. Later on, Caffarelli, Gidas and Spruck \cite{CaffarelliGidasSpruck} removed the growth assumption by introducing the Kelvin transformation and got the same results. In the case that $1\leq \alpha<\frac {n+2}{n-2}$, Gidas and Spruck \cite{GidasSpruck} showed that \eqref{intro1} admitted only trivial solution.
\par An interesting related problem is the extension of \eqref{intro1} to the degenerate elliptic case,\begin{equation}
\label{intro2}y\partial_y^2 u+\Delta_x u+a\partial_y u+u^{\alpha}=0,\quad (x,y)\in \mathbb R^n\times \mathbb R^+=\mathbb R^{n+1}_+, \quad n\geq 1,
\end{equation}here $a\geq 1$ is a constant. Such an equation arises from the isometric embedding of Alexandrov-Nirenberg surfaces when we dealt with the a priori estimates of the second fundamental forms, see \cite{HanHongHuang}. In \cite{Huang1}, the author got that for $1<\alpha<\frac{n+2a+2}{n+2a-2}$, the only nonnegative solutions for \eqref{intro2} is 0 and classified all the nonnegative solutions for $\alpha=\frac{n+2a+2}{n+2a-2}$.
\par Then, it's natural to consider the high order degenerate elliptic case.
we consider the nonnegative solutions $u\in C^{2p}(\overline{\mathbb R^{n+1}_+})$ of the following high order degenerate elliptic equations\begin{equation}
\label{intromain}
(-A)^p u=u^\alpha\text{ in } \mathbb R^{n+1}_+,n\geq 1,
\end{equation}where the operator $A=y\partial_y^2+\Delta_x+a\partial_y,1\leq p< \frac {n+2a}2,p\in\mathbb Z$ and $a\geq 1$ is a constant.
As is known, there are many high order elliptic extension results concerning \eqref{intro1}, for instance \cite{ChenLiOu, LiYY, Lin, LiuGuoZhang2006, WeiXu, Xu} and references therein. Inspired by these results, we have the following theorem of \eqref{intromain}.\begin{theorem}\label{mainthm1}
Let $0\leq u(x,y)\in C^{2p}(\overline{\mathbb R^{n+1}_+})$
satisfy the following equation,
\begin{equation}\label{intromain1}(-A)^p u=u^{\alpha},\quad \text{in}\quad \mathbb R^{n+1}_+,\quad p\in \mathbb Z^+,2p<n+2a,\end{equation}where the operator $A=y\partial_y^2+a\partial_y+\Delta_x$, $a\geq 1$ is a constant. Then \begin{itemize}
\item[(1)] for $1<\alpha<\frac{n+2a+2p}{n+2a-2p}$, $u\equiv 0$;
\item[(2)] for $\alpha=\frac{n+2a+2p}{n+2a-2p}$, $\displaystyle u_{t,x_0}(x,y)=
c_0\left(\frac{t}{t^2+4y+|x-x_0|^2}\right)^{\frac{n+2a-2p}{2}}$
\end{itemize}for some $x_0\in \mathbb R^n$ and $t>0$.
\end{theorem}We will prove Theorem \ref{mainthm1} by the method of moving plane.
Noting the classification in Theorem \ref{mainthm1}, we think $x,y$ play different scales in the equation \eqref{intromain1}. So in fact, we always take the transformation $x_{n+1}=2\sqrt y$ to make \eqref{intromain1} easier to  be dealt with. After the transformation $x_{n+1}=2\sqrt y$, \eqref{intromain1} changes to\begin{equation}
\label{intromain2}(-\widetilde\Delta_{n+1,a})^p u=u^{\alpha},\quad \text{in}\quad \mathbb R^{n+1}_+,\quad p\in \mathbb Z^+,2p<n+2a,
\end{equation}where the operator $\displaystyle \widetilde\Delta_{n+1,a}=\sum_{i=1}^{n+1}\partial^2_{x_i}+\frac{2a-1}{x_{n+1}}\partial_{x_{n+1}}$.
\par In order to apply the moving plane method to the high order elliptic cases of \eqref{intro1}, one important step  is to prove that $(-\Delta)^i u>0,i=1,\cdots,p-1$. Similarly, for \eqref{intromain2}, we have \begin{theorem}
\label{mainthm2}Let $0< u(x)\in C^{2p}(\mathbb R^{n+1})$ be an even function with respect to $x_{n+1}$ and
satisfy \eqref{intromain2} in $\mathbb R^{n+1}$. Then it's valid that\begin{equation}
(-\widetilde\Delta_{n+1,a})^i u>0,\quad\text{in}\quad \mathbb R^{n+1},i=1,\cdots,p-1.
\end{equation}
\end{theorem} In getting Theorem \ref{mainthm2}, we mainly follow the arguments in \cite{WeiXu}. Since the appearance of the first order derivative, we can't just use the sphere average  $\bar u(r)=\frac 1{|\partial B_r|}\int_{\partial B_r} u dS$, instead, we use weighted average
$$\bar u_w(r)=\frac 1{r^{n+2a-1}}\int_{\partial B_r} |x_{n+1}|^{2a-1}u(x) dS.$$
This is the new idea in the present paper and causes some changes in the proof.
 \par In the proof of Theorem \ref{mainthm1}, we inevitably encounter with the maximal principle for $u\in C^{2p}(B_1\backslash\{0\})$ even with respect to $x_{n+1}$ and \begin{equation}\label{intro3}
(-\widetilde\Delta_{n+1,a})^p u=|x|^{-\tau}u^{\alpha}, \quad\text{in}\quad B_1\backslash \{0\}\subset \mathbb R^{n+1},\alpha>1,\tau=(n+2a+2p)-\alpha(n+2a-2p).
\end{equation} In \cite{WeiXu}, the authors proved that $|x|^{-\tau}u^\alpha\in L^1(B_1)$. Then by the maximum principle for weak supper harmonic functions, they could show $(-\Delta )^i u\geq \inf_{\partial B_r} u$. In our case, we can only prove $|x_{n+1}|^{2a-1}|x|^{-\tau}u^\alpha \in L^1(B_1)$. And also, we don't have the corresponding maximum principle. Instead of this,  we establish a Green formula for \eqref{intro3} overcome this difficulty. This idea is originated from \cite{Li1996} for Laplacian equation generalized by \cite{ChernYang} for polyharmonic operators. We extend these results to some weighted divergence system. This is also an interesting part of our paper.
\begin{theorem}\label{mainthm3}
Let $u\in C^{2p}(\bar B_1\backslash\{0\})$ with $u(x',x_{n+1})=u(x',-x_{n+1})$ satisfy \eqref{intro3}. Moreover, we assume that $|x_{n+1}|^{2a-1}|x|^{-\tau} u^{\alpha}$ $\in L^1(B_1)$, then \begin{equation}
\label{intromain3}\int_{\partial B_1}\left[\frac{\partial v_i}{\partial r}-(2a-1)v_i\right]dS+\lim_{s\rightarrow 0}\int_{B_1\backslash B_s}v_{i+1}dx=0,
\end{equation}here $v_i=|x_{n+1}|^{2a-1}(-\widetilde\Delta_{n+1,a})^i u,i=0,\cdots,p-1$.
\end{theorem}
\par The present paper is organized as follows. In Section 2, we will collect some preliminary results concerning about the maximal principles and the asymptotic properties. In Section 3, we will prove the ``superharmonic" property of $(-\widetilde\Delta_{n+1,a})^i u$ or Theorem \ref{mainthm2}. This section mainly follows the arguments of \cite{WeiXu} except for the utility of weighted spherical average. We will establish a divergence identity in Section 4. This is a generalization of the works of \cite{Li1996} and \cite{ChernYang}. Also there are some interesting ideas in both Section 3 and Section 4. The last section is devoted to prove Theorem \ref{mainthm1}.

\section{Preliminary results}
\setcounter{section}{2} \setcounter{equation}{0}
\setcounter{theorem}{0}\setcounter{lemma}{0}
In this section, we first collect some basic facts. \begin{lemma}
\label{prelem1}If $u(x,y)\in C^{2k}(\overline{\mathbb R^{n+1}_+})$ and $v(x,t)=u(x,\frac{t^2}4)$, then we have $v(x,t)\in C^{2k}(\overline{\mathbb R^{n+1}_+})$ and \begin{equation}
\label{pre1}\frac{\partial^{2l-1} v(x,t)}{\partial t^{2l-1}}|_{t=0}=0,\quad l=1,2,\cdots,k.\end{equation}\end{lemma}
\begin{proof}
It is obvious that $v(x,t)\in C^{2k}(\overline{\mathbb R^{n+1}_+})$. We only need to show \eqref{pre1} is true. For $l=1$, one has $\frac{\partial v}{\partial t}=\frac t2\frac{\partial u}{\partial y}$ for $t=2\sqrt y$. We prove \eqref{pre1} by induction. Suppose for $l$ we have \begin{equation}\label{pre2}\frac{\partial^{2l-1} v}{\partial t^{2l-1}}=\sum_{k=1}^{2l-1}\sum_{i=1}^lc_{ik,l}t^{2i-1}\frac{\partial^k u}{\partial y^k},\quad \frac{\partial^{2l} v}{\partial t^{2l}}=\sum_{k=1}^{2l}\sum_{i=0}^lc'_{ik,l}t^{2i}\frac{\partial^k u}{\partial y^k}\end{equation} for some constants $c_{ik,l},c'_{ik,l}$. Then for $l+1$, one can see\begin{eqnarray*}
&&\frac{\partial^{2l+1} v}{\partial t^{2l+1}}=\sum_{k=1}^{2l}\sum_{i=0}^l2ic'_{ik,l}t^{2i-1}\frac{\partial^k u}{\partial y^k}+\sum_{k=1}^{2l}\sum_{i=0}^l\frac 12c'_{ik,l}t^{2i+1}\frac{\partial^{k+1} u}{\partial y^{k+1}}=\sum_{k=1}^{2l+1}\sum_{i=1}^{l+1} c_{ik,l+1}t^{2i-1}\frac{\partial^k u}{\partial y^k}\\
&&\frac{\partial^{2l+2} v}{\partial t^{2l+2}}=\sum_{k=1}^{2l+1}\sum_{i=1}^{l+1}(2i-1)c_{ik,l+1}t^{2i-2}\frac{\partial^k u}{\partial y^k}+\sum_{k=1}^{2l+1}\sum_{i=1}^{l+1}\frac 12 c_{ik,l+1}t^{2i}\frac{\partial^{k+1} u}{\partial y^{k+1}}=\sum_{k=1}^{2l+2}\sum_{i=0}^{l+1} c'_{ik,l+1}t^{2i}\frac{\partial^k u}{\partial y^k}
\end{eqnarray*} This proves \eqref{pre2}. By \eqref{pre2}, it is easy to see that $\frac{\partial^{2l-1} v(x,t)}{\partial t^{2l-1}}|_{t=0}=0,\quad l=1,2,\cdots,k$.
\end{proof}
From Lemma \ref{prelem1}, we have the following remark.
\begin{remark}\label{prerem1}
For any $u(x,y)\in C^{2p}(\overline{\mathbb R^{n+1}_+})$, if we take transformation $x_{n+1}=2\sqrt y$ and $\tilde u(x,x_{n+1})=u(x,y)$. Then we can take the even extension $\tilde u(x,x_{n+1})=\tilde u(x,-x_{n+1})$ such that $\tilde u(x,x_{n+1})\in C^{2p}(\mathbb R^{n+1})$. This means \eqref{intromain2} is still true in $\mathbb R^{n+1}$ if we take even extension with respect to $x_{n+1}$.
\end{remark}
By Lemma \ref{prelem1} and Remark \ref{prerem1}, we shall always consider \eqref{intromain2} in $\mathbb R^{n+1}$ by even extension.
%Combining Lemma \ref{prelem1} and Remark \ref{prerem1}, we now consider \eqref{intromain2} as follows,\begin{equation}
%
%\end{equation}
\par In order to use the moving plane method, we need some new maximal principles.  Consider the following
elliptic operator
$$L(u)=\displaystyle \sum_{i=1}^{n+1} a_{ij}(x)\partial_{ij}u+\sum_{i=1}^n
b_i(x) \partial_iu+\frac{a(x)}{x_{n+1}}\partial_{n+1}u.$$ All the
coefficients $a_{ij}(x),b_i(x),a(x)\in C(\mathbb R^{n+1})$, $a(x)\geq  0$
and $(a_{ij})$ is a positive definite matrix. Then we shall have
the following two lemmas. Let $B_1$ be the unit ball centered at origin.
\begin{lemma}\label{prelem2} Suppose that $u\in
C^2(B_1)\cap C(\bar{B}_1)$ with $\partial_{n+1}u(x',0)=0$ satisfies
that\begin{equation*} -L(u)\geq 0 \text{ in } B_1.
\end{equation*} Then either $u$ is a constant or $u$ can not attain its minimum in
$B_1$.
\end{lemma}
\begin{lemma}\label{prelem3}
Suppose that $u(x)\in C^2(B_1)\cap C^1(\bar{B}_1)$ with
$\partial_{n+1}u(x',0)=0$ satisfies that\begin{equation} -L(u)\geq 0
\text{ in } B_1.
\end{equation}
If $u$ attains its minimum at $x^0\in \partial B_1$, then either
$u\equiv const$ or $$-\frac{\partial u}{\partial n}|_{x=x^0}>0,\ n
\text{ is the outward normal to $\partial B_1$ at } x^0.$$
\end{lemma}
Lemma \ref{prelem2} and Lemma \ref{prelem3} are obtained in \cite{Huang1}, we omit the proof here. Also, we need the following lemma for a punctured ball which was also proved in \cite{Huang1}.
\begin{lemma}\label{lem005}
Suppose that $v\in C^2(B_1\backslash \{0\})\cap C(\bar{B}_1\backslash\{0\})$ is a solution
of  the following problem with $n+2a>2$,
\begin{equation}\Delta_{n+1} v+\frac{2a-1}{x_{n+1}}\partial_{n+1}v\leq 0, \quad \text{in } B_1\backslash\{0\},\quad \partial_{n+1}v(x',0)=0.
\end{equation}If we have  $\displaystyle \underset{x\rightarrow 0}{\underline{\lim}}|x|^{n+2a-2}v(x)\geq 0$, then there holds $$v(x)\geq \underset{\partial B_1}{\inf}
v, \ \   \   \   \   \   \   \   \   \forall x\in B_1\backslash
\{0\}.$$
\end{lemma}

Also we have the following Kelvin transformation,
\begin{equation}\label{pre3}
u^*(x)=|x|^{2p-n-2a}u\left(\frac x{|x|^2}\right).
\end{equation}
If $u(x)$ satisfies \eqref{intromain2}, then we have
\begin{equation}
\label{pre4}(-\widetilde\Delta_{n+1,a})^p u^*=|x|^{-\tau}\left(u^*\right)^{\alpha},\quad \text{in}\quad \mathbb R^{n+1}\backslash\{0\},\tau=(n+2a+2p)-\alpha(n+2a-2p).\end{equation}
As the proof of \eqref{pre4} is of independent interest, we will present it  in the Appendix.
%{\bf The proof of \eqref{pre4}}: First we need to show in the sphere coordinates $(r,\theta_1,\cdots,\theta_n)$, \begin{equation}\label{pre5}\widetilde\Delta_{n+1,a}=\partial_r^2+\frac{n+2a-1}{r}\partial_r+\frac 1{r^2}\mathscr L_{\theta},
%\end{equation} here $\mathscr L_{\theta}$ is some differential operator with respect to $\theta=(\theta_1,\cdots,\theta_n)$. One has the following chain rule,\begin{eqnarray}\label{pre6}\frac{\partial}{\partial {x_{n+1}}}=\frac{\partial r}{\partial {x_{n+1}}}\frac{\partial}{\partial r}+\sum_{i=1}^n \frac{\partial \theta_i}{\partial {x_{n+1}}}\frac{\partial}{\partial \theta_i}=\frac{x_{n+1}}{r}\frac{\partial}{\partial r}+\sum_{i=1}^n \frac{l_i(\theta)}{r}\frac{\partial}{\partial \theta_i},\end{eqnarray} for some smooth function $l_i(\theta)$. Then we have \begin{eqnarray*}
%\widetilde\Delta_{n+1,a}&=&\partial_r^2+\frac{n}{r}\partial_r+\frac 1{r^2}\Delta_{\mathbb S^n}+\frac{2a-1}r\partial_r+\frac 1{r^2}\sum_{i=1}^n \frac{(2a-1)l_i(\theta)}{\cos \theta_1}\frac{\partial}{\partial \theta_i}\\&=&\partial_r^2+\frac{n+2a-1}{r}\partial_r+\frac 1{r^2}\mathscr L_{\theta},
%\end{eqnarray*}if we take $x_{n+1}=r\cos \theta_1$, $\Delta_{\mathbb S^n}$ as the Laplace-Beltrimi operator on $\mathbb S^n$.
\begin{lemma}
\label{prelem4}
Let $u(x)\in C^\infty(\mathbb R^{n+1})$ be an even function with respect to $x_{n+1}$ and $u^*(x)$ be defined in \eqref{pre3}. Then \begin{equation}\label{pre7}
(-\widetilde\Delta_{n+1,a})^i u^*(x)=\frac{c_i}{|x|^{n+2a-2p+2i}}f_i\left(\frac x{|x|^2}\right),\quad i=0,\cdots,p-1\end{equation} for some constants $c_i>0$ and smooth functions $f_i(x)$ with $f_i(0)=u(0)$ and $f_i(x)$ are even functions with respect to $x_{n+1}$. Moreover, we have $(-\widetilde \Delta_{n+1,a})^i u^*(x)>0$ for $|x|$ large enough.
\end{lemma}
\begin{proof}
We shall prove \eqref{pre7} by induction. For $i=0$, $c_0=1,f_0(x)=u(x)$. Now \begin{eqnarray*}
(-\widetilde\Delta_{n+1,a})^{i+1} u^*(x)&=&(-\widetilde\Delta_{n+1,a})\left(\frac{c_i}{|x|^{n+2a-2p+2i}}f_i\left(\frac x{|x|^2}\right)\right)\\
&=&\frac{c_i}{|x|^{n+2a-2p+2(i+1)}}\left\{(2p-2i-2)(n+2a-2p+2i)f_i\left(\frac{x}{|x|^2}\right)\right.\\
&&+\left.4(p-i-1)\sum_{j=1}^{n+1}\frac{x_j}{|x|^2}\frac{\partial f_i}{\partial x_j}\left(\frac{x}{|x|^2}\right)-|x|^{-2}(\widetilde\Delta_{n+1,a} f_i)\left(\frac{x}{|x|^2}\right)\right\}.
\end{eqnarray*}
Set \begin{eqnarray*}
&&(2p-2i-2)(n+2a-2p+2i)f_{i+1}(x)\\
&&=(2p-2i-2)(n+2a-2p+2i)f_{i}(x)+4(p-i-1)\sum_{j=1}^{n+1} x_j\frac{\partial f_i}{\partial x_j}(x)-|x|^2\widetilde\Delta_{n+1,a} f_i(x),\\
&&c_{i+1}=(2p-2i-2)(n+2a-2p+2i)c_i
\end{eqnarray*}
Then
$$(-\widetilde\Delta_{n+1,a})^{i+1} u^*(x)=\frac{c_{i+1}}{|x|^{n+2a-2p+2(i+1)}}f_{i+1}\left(\frac x{|x|^2}\right)$$
and
$$f_{i+1}(0)=f_i(0)=u(0),f_{i+1}(x',x_{n+1})=f_{i+1}(x',-x_{n+1}).$$
As $c_i>0$ for $i=0,\cdots,p-1$,$$(-\widetilde\Delta_{n+1,a})^i u^*(x)=\frac{c_iu(0)}{|x|^{n+2a-2p+2i}}\left(1+O\left(\frac1{|x|}\right)\right)>0$$ for $|x|$ large enough.
\end{proof}

\section{The ``superharmonic" property}
\setcounter{section}{3} \setcounter{equation}{0}
\setcounter{theorem}{0}\setcounter{lemma}{0}
%Denote
%$$A(u)=\sum_{i=1}^m y_iu_{y_iy_i}+ \sum_{i=1}^{m}a_i
%u_{y_i}+\Delta_x u$$ We shall also have the following theorem
%
%\begin{theorem}
%\label{thm3} Let $0\leq u(x,y)\in C^{2p}(\overline{R^{m,n}_+})$
%satisfy the following equation,
%
%$$ (-A)^p u=u^{\alpha} \text{ in }R^{m,n}_+,\   p\in Z^+,2p<n+2a, $$ where the operator $A=y\partial_y^2+a\partial_y+\Delta_x$. Then \begin{itemize}
%\item[(1)] for $1<\alpha<\frac{n+2a+2p}{n+2a-2p}$, $u\equiv 0$
%\item[(2)] for $\alpha=\frac{n+2a+2p}{n+2a-2p}$, $u_{t,x_0}(x,y)=
%c_0t^{\frac{n+2a-2p}{2}}(t^2+4\displaystyle\sum_{i=1}^m
%y_i+|x-x_0|^2)^{-\frac{n+2a-2p}{2}}$
%\end{itemize}for some $x_0\in R^n$ and $c_0,t>0$.
%\end{theorem}
%
% As the arguments in
%Section 3, we only need to show the proof for the case $m=1$.
 %Set
%$$\tilde{\Delta}=\Delta +\frac{2a-1}{x_{n+1}}\partial_{n+1}.$$ Then
%we have the following lemma
%\begin{lemma}\label{lem501} If $u(x)\in C^{2p}(R^{n+1})$ with
%$u(x',x_{n+1})=u(x',-x_{n+1})$ is a  solution of
%\begin{equation}
%\begin{cases}
%(-\widetilde\Delta_{n+1,a} )^p u=u^\alpha \text{ in } R^{n+1}\\
%u>0\text{ in } R^{n+1},\alpha>1
%\end{cases}
%\end{equation}
%it follows that \begin{equation} (-\widetilde\Delta_{n+1,a})^i u>0,\    i=1,2,...,p-1
%\end{equation}
%\end{lemma}
This section is devoted to prove Theorem \ref{mainthm2}.
\par {\bf The proof of Theorem \ref{mainthm2}}:
\par Denote $u_i=(-\widetilde\Delta_{n+1,a})^i u$, $i=0,1,2...,p-1$ with $u_0=u$. We first
prove  $$u_{p-1}>0.$$

If not, there exists $x_0\in \mathbb R^{n+1}$ such that $$u_{p-1}(x_0)<0.$$
For simplicity we can just assume $x_0=0$. We need make some explanations here. As $\widetilde \Delta_{n+1,a}$ is invariant under the translation of $x_1,\cdots,x_n$, we can translate $x_1,\cdots,x_n$ to make the first $n$ coordinates to be $0$. Hence we may assume $x_0=(0,\cdots,0,b)$ for some $b\geq 0$. Now we make a transformation $\tilde u(x',x_{n+1},x_{n+2})=u(x',\sqrt{x_{n+1}^2+x_{n+2}^2})$ where $x'=(x_1,\cdots,x_n)$. By directly computation, we can derive that
$$(-\tilde\Delta_{n+2,a-\frac 12})^p \tilde u =\tilde u^\alpha,\quad \text{ in }\mathbb R^{n+2},$$
where $-\tilde\Delta_{n+2,a-\frac 12}=\sum_{i=1}^{n+2}\partial_{x_i}^2+\frac{2a-2}{x_{n+2}}\partial_{x_{n+2}}$. By the transformation, we have
$$\tilde u_{p-1}(x',c,d)=(-\tilde\Delta_{n+2,a-\frac 12})^{p-1}\tilde u(x',c,d)=(-\tilde\Delta_{n+1,a})^{p-1} u(x',b)=u_{p-1}(x',b), \text{ for }b=\sqrt{c^2+d^2}.$$
Set $x'=0, c=b, d=0$, we have for $\tilde u_{p-1}(0,b,0)<0$. Since $\tilde u $ is invariant under translation with respect to $x_1,\cdots,x_{n+1}$, we have after translation $\tilde u_{p-1}(0)<0$. Therefor, by the above arguments, we can just assume $x_0=0$, otherwise we can consider $\tilde u$.
  Set $$
\bar{f}(r)=\frac{1}{|\partial B_r(0)|}\int_{\partial B_r(0)}f dS.$$
From now on, if no confusion occurs, we will always normalize $|\mathbb S^n|=1$ where $\mathbb S^n$ is the unit sphere in $\mathbb R^{n+1}$.
Then for $u_{p-1}$, we have \begin{eqnarray*}
-\int_{B_r}|x_{n+1}|^{2a-1}u^\alpha dx&=&\int_{B_r}\nabla\cdot(
|x_{n+1}|^{2a-1}\nabla u_{p-1})dx\\&=& \int_{\partial B_r}
|x_{n+1}|^{2a-1}\frac{\partial u_{p-1}}{\partial \rho} dS
\\&=&\int_{\partial B_r}\frac{\partial(|x_{n+1}|^{2a-1}u_{p-1})}{\partial
\rho}-(2a-1)\frac{|x_{n+1}|^{2a-1}u_{p-1}}{r}dS,\quad \rho=|x|.
\end{eqnarray*}
In getting the last equality, we have used $$\frac{\partial|x_{n+1}|^{2a-1}}{\partial \rho}=(2a-1)\frac{|x_{n+1}|^{2a-1}}{\rho},\quad \rho=|x|.$$
Taking $w_{p-1}(x)=|x_{n+1}|^{2a-1}u_{p-1}(x)$, one gets
 that $$(r^n\bar{w}'_{p-1}-(2a-1)r^{n-1}\bar{w}_{p-1})'\leq 0.$$
This implies that
$$r^n\bar{w}'_{p-1}-(2a-1)r^{n-1}\bar{w}_{p-1}< 0\Rightarrow \frac{
\partial(r^{-(2a-1)}\bar{w}_{p-1})}{\partial r}<0,\text{ for } r>0.$$
Set $z_{p-1}(r)=r^{-(2a-1)}\bar{w}_{p-1}(r)$, then
$$z_{p-1}(r)<z_{p-1}(0)<0,\text{ for all } r>r_1=0.$$
$z_{p-1}(0)<0$ follows from\begin{eqnarray*}
z_{p-1}(0)&=&\lim_{r\rightarrow 0}r^{-(n+2a-1)}\int_{\partial B_r}|x_{n+1}|^{2a-1} u_{p-1} dS
\\&=& \lim_{r\rightarrow 0}r^{-n}\int_{\partial B_r}|\cos \theta_1|^{2a-1} u_{p-1} dS,\quad x_{n+1}=r\cos\theta_1
\\&=& u_{p-1}(0)\int_{\partial B_1} |\cos \theta_1|^{2a-1} dS<0.
\end{eqnarray*}
Repeating the above steps, it is easy to see that
$$z_{p-2}'(r)>-c_1z_{p-1}(0)r.$$
Hence
$$z_{p-2}(r)\geq c_2r^2,\text{ for }r\geq r_2>r_1.$$
By induction, it follows that
$$ (-1)^iz_{p-i}(r)\geq c_ir^{2(i-1)}\text{
for }r\geq r_i,i=1,\cdots,p,\quad \text{for}\quad c_i>0.$$
Hence if $p$ is odd, it's a contradiction
with $u>0$.

So $p$ must be even such that $$z_0(r)\geq c_0r^{\sigma_0},
\sigma_0=2(p-1)
$$ and $$(-1)^iz_{p-i},(-1)^iz'_{p-i}>0,i=1,\cdots,p,\text{ for } r>r_0>0.$$

We set $A=2\alpha(p-1)+n+2a+2p$ and assume that $$ z_0(r)\geq
\frac{c_0^{\alpha^k}r^{\sigma_k}}{A^{b_k}},\forall r\geq r_k.$$
Integrating by parts, one gets
\begin{equation}\begin{split} &\int_{\partial B_r}
\left(\frac{\partial\left(|x_{n+1}|^{2a-1}u_{p-1}\right)}{\partial
\rho}-
(2a-1)\frac{|x_{n+1}|^{2a-1}u_{p-1}}{r}\right)dS\\-&\int_{\partial
B_{r_k}}
\left(\frac{\partial\left({|x_{n+1}|^{2a-1}u_{p-1}}\right)}{\partial
\rho}-(2a-1)\frac{|x_{n+1}|^{2a-1}u_{p-1}}{r_k}\right)dS+
\int_{B_r\backslash
B_{r_k}}|x_{n+1}|^{2a-1}u^{\alpha}dx=0.\label{a402}\end{split}
\end{equation}
In fact, we have \begin{eqnarray*} \int_{B_r\backslash
B_{r_k}}|x_{n+1}|^{2a-1}u^{\alpha}dx&=&\int_{r_k}^r
d\rho\int_{\partial
B_{\rho}}(|x_{n+1}|^{2a-1}u)^{\alpha}|x_{n+1}|^{(1-\alpha)(2a-1)}dS\\
&\geq &\int_{r_k}^r \rho^{n+(1-\alpha)(2a-1)}\bar{w}_0^\alpha
d\rho\\&=&\int_{r_k}^r \rho^{n+2a-1}z_0^\alpha d\rho
\end{eqnarray*}In getting the second inequality, we have used $\alpha>1$, $|x_{n+1}|<\rho$ and the convexity of $g(t)=t^\alpha$.
Therefore, \eqref{a402} means that \begin{equation}\label{444} \begin{split}
r^{n+2a-1}z_{p-1}'(r)&\leq r_k^{n+2a-1}z_{p-1}'(r_k)-\int_{r_k}^r
\rho^{n+2a-1}z_0^\alpha d\rho\\&\leq -c_0^{\alpha^{k+1}}
\frac{r^{\alpha\sigma_k+n+2a}-r_k^{\alpha\sigma_k+n+2a}}{A^{\alpha b_k}(
\alpha\sigma_k+n+2a)}\\&\leq
-\frac{c_0^{\alpha^{k+1}}r^{\alpha\sigma_k+n+2a}}{2A^{\alpha b_k}(
\alpha\sigma_k+n+2a)}
\end{split}
\end{equation}for $r\geq 2^{\frac{1}{\alpha\sigma_k+n+2a}}r_k$, since $z'_{p-1}(r_k)<0$.
Hence
$$z_{p-1}'(r)\leq -\frac{c_0^{\alpha^{k+1}}r^{\alpha\sigma_k+1}}{2A^{\alpha b_k}(\alpha\sigma_k+n+2a)}.$$
Then
 \begin{eqnarray*}z_{p-1}(r)&\leq&
-\frac{c_0^{\alpha^{k+1}}r^{\alpha\sigma_k+2}}{4A^{\alpha b_k}(
\alpha\sigma_k+n+2a)(\alpha\sigma_k+2)}\\
&\leq &-\frac{c_0^{\alpha^{k+1}}r^{\alpha\sigma_k+2}}{4A^{\alpha b_k}(
\alpha\sigma_k+n+2a)^2}\end{eqnarray*} for $r\geq
2^{\frac{2}{\alpha\sigma_k+1}}r_k$. By induction, it follows that$$
(-1)^iz_{p-i}(r)\geq
\frac{c_0^{\alpha^{k+1}}r^{\alpha\sigma_k+2i}}{A^{\alpha
b_k}4^i(\alpha\sigma_k+n+2a+2p)^{2i}},r\geq
2^{\frac{2i}{\alpha\sigma_k+1}}r_k.$$ Especially, $$z_0(r)\geq \frac{
c_0^{\alpha^{k+1}}r^{\alpha\sigma_k+2p}}{A^{\alpha
b_k}4^p(\alpha\sigma_k+n+2a+2p)^{2p}},r\geq
2^{\frac{2p}{\alpha\sigma_k+1}}r_k.$$ Set $\sigma_0=2(p-1)$, $r_0$,
then
$$\sigma_{k+1}=\alpha\sigma_k+2p,r_{k+1}=2^{\frac{2p}{\alpha\sigma_k+1}}r_k.$$
First of all, by mathematical induction, it is easy to see that
$$A^{\alpha b_k}4^p(\alpha\sigma_k+n+2a+2p)^{2p}\leq A^{2p(k+1)+\alpha b_k}$$
if we notice that
$$2(\alpha\sigma_k+n+2a+2p)\leq A(\alpha\sigma_{k-1}+n+2a+2p).$$
 Thus, we also can set
$$b_0=0,b_{k+1}=\alpha b_k+2p(k+1).$$

Then we have
$$z_0(r)\geq \frac{c_0^{\alpha^{k+1}}r^{\sigma_{k+1}}}{A^{b_{k+1}}},r\geq r_{k+1}.$$
It's very important that we shall notice that
$$r_{k+1}\leq cr_0$$ where $c$ can be chosen to be
$2^{\sum_{k=0}^\infty\frac{2p}{\alpha\sigma_k+1}}$.

By a direct computation, we have
$$\sigma_k=2(p-1)\alpha^k+\frac{2p(\alpha^k-1)}{\alpha-1},b_k=2p\left(\frac{\alpha(\alpha^k-1)}{(\alpha-1)^2}-\frac k{\alpha-1}\right).$$

Taking $\bar r=\max(\frac{2A^{\frac{2p\alpha}{(\alpha-1)^2}}}{c_0},cr_0)$, we will have
$$z_0(\bar r)\geq\frac{c_0^{\alpha^{k+1}}\bar r^{\sigma_{k+1}}}{A^{b_{k+1}}}
 \rightarrow \infty \text{ as } k\rightarrow\infty$$ which is a contradiction. Hence $$u_{p-1}>0.$$
Next we claim that $$u_{p-i}>0,i=2,...,p-1.$$  By induction, we have for $i=1,\cdots,k$, $$u_{p-k}>0.$$
If $u_{p-k-1}(0)<0$, then from $$-\widetilde\Delta_{n+1,a}u_{p-k-1}=u_{p-k}>0,$$ following the same arguments as $k=1$, we have $ z_{p-k-1}(r)< z_{p-k-1}(0)<0$. Also $$(-1)^{i-k} z_{p-i}\geq c_1 r^{2(i-k-1)},\quad \text{for}\quad r\geq r_0,p\geq i\geq k+1.$$ If $p-k$ is odd, it is a contradiction to $ z_0>0$. Then $p-k$ must be even, this means $ z_0(r)\geq cr^{2(p-k-1)}\geq c r^2$ for $r\geq r_0>1$. By \eqref{444}, one can see \begin{equation} \begin{split}
r^{n+2a-1}z_{p-1}'(r)&\leq r_0^{n+2a-1}z_{p-1}'(r_0)-\int_{r_0}^r
\rho^{n+2a-1}z_0^\alpha d\rho\\&\leq r_0^{n+2a-1}z_{p-1}'(r_0) -
\frac{c^{\alpha} (r^{2\alpha+n+2a}-r_0^{2\alpha+n+2a})}{(
2\alpha+n+2a)}\\&\leq
-\frac{c^{\alpha}r^{2\alpha+n+2a}}{2(
2\alpha+n+2a)}
\end{split}
\end{equation}for $r>r_1\geq r_0$ since $r_0^{n+2a-1}z_{p-1}'(r_0)$ is bounded. This means \begin{equation*}
z_{p-1}(r)\leq -\frac{c^{\alpha}r^{2\alpha+2}}{4(n+2a+2\alpha)}
\end{equation*}for $r\geq r_2>r_1$ which contradicts to $z_{p-1}>0$. This ends the proof of Theorem \ref{mainthm2}.

\section{A divergence identity in a punctured domain}
In order to prove Theorem \ref{mainthm3}, we begin with the following lemma.
\begin{lemma}\label{div}
If all the assumptions of Theorem \ref{mainthm3} are satisfied, then for $i=0,\cdots, p$
\begin{equation}
\left|\lim_{s\rightarrow 0}\int_{B_1\backslash B_s}v_{i}dx\right|< \infty, \quad \text{here}\quad v_i=|x_{n+1}|^{2a-1}(-\widetilde\Delta_{n+1,a})^i u.
\end{equation}
\end{lemma}
\begin{proof}
We only need to verify for $i=p-1$. Set $g(x)=|x_{n+1}|^{2a-1}|x|^{-\tau}u^\alpha$. Then $g\in L^1(B_1)$ by assumption.
\begin{equation*}
\begin{split}
&\int_{B_1\backslash B_r}|x_{n+1}|^{2a-1}(-\widetilde \Delta_{n+1,a})^p u dx=\int_{B_1\backslash B_r}g dx\\
\Rightarrow & -\int_{\partial B_1}\left(\frac{\partial v_{p-1}}{\partial\rho}-(2a-1)v_{p-1}\right) dS+\int_{\partial B_r} \left(\frac{\partial v_{p-1}}{\partial\rho}-(2a-1)r^{-1}v_{p-1}\right)dS=|g|_{L^1(B_1)}+o(1)
\\\Rightarrow & r^{n}\bar v_{p-1}'-(2a-1)r^{n-1}\bar v_{p-1}=|g|_{L^1(B_1)}+\bar v_{p-1}'(1)-(2a-1)\bar v_{p-1}(1)+o(1)=\beta_{p-1}+o(1)\\
\Rightarrow & (r^{-(2a-1)}\bar v_{p-1})'=r^{-(n+2a-1)}(\beta_{p-1}+o(1))\\
\Rightarrow & r^{n-1}\bar v_{p-1}(r)=-(n+2a-2)^{-1}\beta_{p-1}+o(1)
\end{split}
\end{equation*}
In the above equations, we always set $\beta_{p-1}=|g|_{L^1(B_1)}+\bar v_{p-1}'(1)-(2a-1)\bar v_{p-1}(1)$.
Then we get \begin{eqnarray*}
\left|\int_{B_1\backslash B_r}v_{p-1}dx\right|=\left|\int_{r}^1d\rho\int_{\partial B_\rho}v_{p-1}dS\right|=\left|\int_r^1\rho^n\bar v_{p-1}d\rho\right|<\infty
\end{eqnarray*}
For the other $v_i$, by induction and repeating the arguments of $v_{p-1}$, we can get the same conclusion.
\end{proof}

With the aid of Lemma \ref{div}, we now can prove Theorem \ref{mainthm3}.
\par {\bf The proof of Theorem \ref{mainthm3}}:
Claim:
\begin{equation}\label{401}
\int_{B_1}|x_{n+1}|^{2a-1}|x|^{-\tau}u^\alpha \varphi dx=\int_{B_1}|x_{n+1}|^{2a-1} u \varphi_pdx,
\end{equation}
 here $\varphi_p$ is defined as follows:
 $$\varphi_0=\varphi(|x|), \varphi_{i}=(-\widetilde \Delta_{n+1,a})^i\varphi_0, \varphi\in C_c^\infty(B_1),i=0,\cdots,p.$$
 Set $\eta(t)\in C^\infty(\mathbb R)$ with $\eta(t)=0,t\le 1$, $\eta(t)=1, t\geq 2$ and $\eta_{\epsilon}(x)=\eta(\frac {|x|}{\epsilon})$.
Taking $\varphi\eta_{\epsilon}$ as the test function, we will
have
\begin{eqnarray*}
&&\int\varphi\eta_{\epsilon}|x_{n+1}|^{2a-1}|x|^{-\tau}u^\alpha dx\\
&=&\int (-\widetilde\Delta_{n+1,a})^{p-1}u(x)L(\varphi\eta_{\epsilon})dx,\quad L(f)=|x_{n+1}|^{2a-1}(-\widetilde\Delta_{n+1,a}) f\\
&= &\int (-\widetilde\Delta_{n+1,a})^{p-1}u(x)(\eta_{\epsilon}L(\varphi)+\varphi L(\eta_{\epsilon})+2|x_{n+1}|^{2a-1}\nabla\varphi\nabla\eta_{\epsilon})dx\\
&=&\int v_{p-1}(\eta_\epsilon \varphi_1+\psi_1)dx, \quad \text{as } v_{p-1}=|x_{n+1}|^{2a-1}(-\widetilde \Delta_{n+1,a})^{p-1} u
\end{eqnarray*} where $\psi_1=-\varphi
\widetilde\Delta_{n+1,a}\eta_{\epsilon}-2\nabla\varphi\nabla\eta_{\epsilon}$ and $supp\psi_1\subset B_{2\epsilon}\backslash B_{\epsilon}, |\psi_1|\leq C_1\epsilon^{-2}$.
We repeat the above process to get \begin{eqnarray*}
&&\int\varphi\eta_{\epsilon}|x_{n+1}|^{2a-1}|x|^{-\tau}u^\alpha dx\\
&=&\int
v_{p-i}(\eta_\epsilon \varphi_i+\psi_i)dx\\
&=&\int
|x_{n+1}|^{2a-1}u(x)(\eta_\epsilon \varphi_p+\psi_p)dx,
\end{eqnarray*} here $\psi_{i+1}=-\widetilde\Delta_{n+1,a}\psi_{i}-\varphi_i
\widetilde\Delta_{n+1,a}\eta_{\epsilon}-2\nabla\varphi_i\nabla\eta_{\epsilon},i=1,\cdots,p-1$. By induction, it is easy to see that $$|\psi_i|\leq C_i\epsilon^{-2i},supp\psi_i\subset B_{2\epsilon}\backslash B_{\epsilon},i=1,\cdots,p.$$
Also we have
\begin{eqnarray*} && \int
u(x)|x_{n+1}|^{2a-1}|\psi_p(x)|dx\\&\leq&
C\left(\int_{B_1}|x|^{-\tau}|x_{n+1}|^{2a-1}u^\alpha\right)^{\frac
1\alpha}\left(\int
|x_{n+1}|^{2a-1}|x|^{\frac{\alpha'\tau}{\alpha}}|\psi_p|^{\alpha'}
\right)^{\frac 1{\alpha'}}\\&\leq
&C\epsilon^{\frac{2p}{\alpha}}\left(\int_{B_1}|x|^{-\tau}
|x_{n+1}|^{2a-1}u^\alpha\right)^{\frac 1\alpha}\rightarrow 0
\end{eqnarray*} as $\epsilon\rightarrow 0$. Let $\epsilon\rightarrow 0$, the claim \eqref{401} follows immediately. From the arguments of  Lemma \ref{div}, one can see that
\begin{eqnarray}\label{average}
&& r^n\bar v_{i}'-(2a-1)r^{n-1}\bar v_{i}=\beta_{i}+o(1),\nonumber\\
&& r^{n-1}\bar v_{i}(r)=-(n+2a-2)^{-1}\beta_{i}+o(1),
\end{eqnarray}where $\beta_i=\lim_{s\rightarrow 0}|v_{i+1}|_{L^1(B_1\backslash B_s)}+\bar v_{i}'(1)-(2a-1)\bar v_{i}(1)$ which is a finite number.
Integrating by parts, we can get\begin{eqnarray}\label{402}
\int_{B_1\backslash B_r}|x_{n+1}|^{2a-1}u\varphi_p dx&=&\int_{B_1\backslash B_r}\varphi |x_{n+1}|^{2a-1}(-\widetilde \Delta_{n+1,a})^p u dx\nonumber\\&-&\sum_{k=0}^{p-1}\int_{\partial B_r}\varphi_k\left(\frac{\partial v_{p-1-k}}{\partial \rho}-(2a-1)r^{-1}v_{p-1-k}\right)+v_{p-1-k}\frac{\partial\varphi_k}{\partial \rho}dS.
\end{eqnarray}
 By the definition of $\varphi_k$, we can see that $\varphi_k$ is  radially symmetric as $\varphi_0$ is radially symmetric. Thus by the average estimates \eqref{average} and Lemma \ref{div}, one can get by taking $r \rightarrow 0$ in \eqref{402},\begin{eqnarray*}
\int_{B_1}|x_{n+1}|^{2a-1}|x|^{-\tau} u^\alpha \varphi_0 dx=\int_{B_1}|x_{n+1}|^{2a-1} u \varphi_p dx+\sum_{i=0}^{p-1}\beta_i\varphi_i(0).
\end{eqnarray*} This means by \eqref{401}, $$\sum_{i=0}^{p-1}\beta_i\varphi_i(0)=0.$$
We can choose $\varphi^{(i)}_0(x)\equiv |x|^{2i},i=0,\cdots,p-1$ for $|x|\leq \frac 12 $ one by one. For $\varphi^{(0)}_0(x)$, we can easily get $\beta_0=0$. Then by induction and choosing suitable $\varphi^{(i)}_0(x)$, we get $\beta_i=0$. It is easy to see that $\beta_i=0$ is equivalent to \eqref{intromain3} holds  if we notice the definition of $\beta_i$.

\section{Nonexistence and classification of positive solutions}

Next we give a lemma without proof which is due to \cite{CMM}.
\begin{lemma}\label{lem502} If $u\in C^{2p}(\mathbb R^{n+1})$, $p\geq 1$ is
radially symmetric and satisfies the inequalities $$(-\widetilde\Delta_{n+1,a})^k u\geq 0,
\text{ in }\mathbb R^{n+1},\quad  k=0,1,\cdots,p$$ where $2p<n+2a$. Then we
have
$$(ru'+(n+2a-2p)u)'<0.$$\end{lemma}

\begin{definition}
\label{har1} Let $l$ be a positive number. We say that a $C^2$ function $f$ has a harmonic asymptotic expansion at infinity in a neighborhood of infinity if:
\begin{equation}
\label{har2}\begin{split}
& f(x)=\frac1{|x|^l}\left(a_0+\sum_{i=1}^{n+1}  \frac{a_ix_i}{|x|^2}\right)+O\left(\frac 1{|x|^{l+2}}\right),
\\& f_{x_i}(x)=-la_0\frac{x_i}{|x|^{l+2}}+O\left(\frac1{|x|^{l+2}}\right),i=1,2,\cdots,n+1,\\
& f_{x_ix_j}=O\left(\frac 1{|x|^{l+2}}\right),i,j=1,\cdots,n+1,
\end{split}
\end{equation}where $a_i\in \mathbb R$ and $a_0>0$.
\end{definition}
Set $$\Sigma_\lambda=\{x\in\mathbb R^{n+1}|x_1<\lambda\},\quad x^\lambda=(2\lambda-x_1,x_2,\cdots,x_{n+1}).$$
\begin{lemma}
\label{move1}Let $f $ be a function in a neighborhood at infinity satisfying the asymptotic expansion \eqref{har2}. Then there exist $\lambda_0>0$ and $R>0$ such that if $\lambda\geq \lambda_0$,$$
f(x)>f(x^\lambda),\quad\text{for}\quad x_1<\lambda,x\notin B_R(0).$$ \end{lemma}

\begin{lemma}
\label{move2}Let $f$ be a $C^2$ positive solution of $-\widetilde\Delta_{n+1,a}f=F(x)$ for $|x|>R$ and $f,F$ be even  functions with respect to $x_{n+1}$, where $f$ has a harmonic asymptotic expansion \eqref{har2} at infinity with $a_0>0$. Suppose that, for some positive number $\lambda_0$ and for every $(x_1,x')$ with $x_1<\lambda_0$,$$
f(x_1,x')>f(2\lambda_0-x_1,x')\quad\text{and}\quad F(x_1,x')\geq F(2\lambda_0-x_1,x').$$Then there exist $\epsilon>0,S>R$ such that
\begin{equation*}
\begin{split}
(i)\quad &f_{x_1}(x_1,x')<0\quad\text{in}\quad |x_1-\lambda_0|<\epsilon,|x|>S,\\
(ii)\quad &f(x_1,x')>f(2\lambda-x_1,x')\quad\text{in}\quad x_1<\lambda_0-\frac 12\epsilon<\lambda,|x|>S,
\end{split}
\end{equation*}for all $x\in\Sigma_{\lambda},\lambda\geq \lambda_1$ with $|\lambda_1-\lambda_0|<c_0\epsilon$, where $c_0$ is a positive number depending on $\lambda_0$ and $f$.
\end{lemma}
Now if $u(x,y)\in C^{2p}(\overline{R^{n+1}_+})$ satisfies \eqref{intromain}, we have $\tilde u(x,x_{n+1})=u(x,\frac{x_{n+1}^2}{4})$ satisfies \eqref{intromain2}. By Lemma \ref{prelem1}, one can extend $\tilde u$ to $\mathbb R^{n+1}$ by $\tilde u(x,x_{n+1})=\tilde u(x,-x_{n+1})$, $x_{n+1}<0$ such that $\tilde u(x,x_{n+1})$ still satisfies \eqref{intromain2} in $\mathbb R^{n+1}$. Define Kelvin transformation as follows
 $$u^*(x)=|x|^{-(n+2a-2p)}\tilde u\left(\frac{x}{|x|^2}\right),$$ then  by a
direct computation, $u^*$ satisfies $$(-\widetilde\Delta_{n+1,a})^p
u^*=|x|^{-\tau}(u^*)^\alpha, \quad \text{in}\quad \mathbb R^{n+1}$$ where
$\tau=n+2a+2p-\alpha(n+2a-2p)\geq 0$, see Appendix for a derivation.

Set $$u^*_{i}=(-\widetilde\Delta_{n+1,a})^{i} u^*(\frac{x}{|x|^2}).$$ Then Lemma \ref{prelem4} tells us that   $u^*_{i}$ has
the asymptotic behavior \eqref{har2}  at $\infty$.
Moreover, we have the following lemma. \begin{lemma}\label{lem503}
$$|x_{n+1}|^{2a-1}|x|^{-\tau}(u^*)^{\alpha}\in L^1(B_1).$$
\end{lemma}
\begin{proof}
If not, by
noting equation \eqref{402}, we set $r=1,r_k=r$ in \eqref{402}. Then
for $r$ small enough, there holds
$$\int_{\partial
B_{r}}
\left(\frac{\partial\left({|x_{n+1}|^{2a-1}u^*_{p-1}}\right)}{\partial
\rho}-(2a-1)\frac{|x_{n+1}|^{2a-1}u^*_{p-1}}{r}\right)dS\leq
-c_1\int_{B_1\backslash
B_{r}}|x_{n+1}|^{2a-1}|x|^{-\tau}(u^*)^{\alpha}dx.$$ Thus we can get,
$$\frac{\partial z_{p-1}}{\partial r}\geq
c_1r^{-(n+2a-1)}\int_{B_1\backslash
B_{r}}|x_{n+1}|^{2a-1}|x|^{-\tau}(u^*)^{\alpha}.$$ Or,
$$z_{p-1}\leq  -c_1r^{-(n+2a-2)},\text{ for }r\leq r_1.$$ Similarly,
$$(-1)^{i}z_{p-i}\geq c_ir^{-(n+2a-2i)},r\leq r_i,i=1,\cdots,p,$$ $z_i(r)$ is defined as $z_i(r)=\displaystyle r^{-(n+2a-1)}\int_{\partial B_r}|x_{n+1}|^{2a-1}u^*_i dS$.
Thus if $p$ is odd, this means $z_0<0$ which is a contradiction.

We only need to consider the case $p$ is even. Then
$$z_1(r)<0,\text{ for }r<r_{p-1},$$ or, $$\widetilde \Delta_{n+1,a} z_0>0.$$
This means that $z_0'(r)<0$ for $r$ small otherwise $z_0$ is
increasing for $r$ small and hence $z_0(r)\leq C$ which contradicts to $z_0\geq c r^{-(n+2a-2p)}$ for $r$ small.

Set
$$z^*(s)=s^{-(n+2a-2p)}z_0\left(\frac 1{s}\right).$$ We shall have $$(-\widetilde\Delta_{n+1,a})
^{p} z^*(s)\geq (z^*)^\alpha(s).$$ By the same arguments as in
the proof of Theorem \ref{mainthm2}, we have $$(-\widetilde\Delta_{n+1,a})^i
z^*(s)>0,i=1,...,p.$$ By Lemma \ref{lem502}, this means$$
(s (z^*)'(s)+(n+2a-2p)z^*(s))'<0.$$ Or,
$$s(z^*)''(s)+(n+2a+1-2p)z^*(s)<0.$$
Now an easy computation yields that $$ \widetilde\Delta_{n+1,a}
z_0(r)=r^{-(n+2a+4-2p)}\left((z^*)''(s)+\frac{n+2a-2p+1}{s}
(z^*)'(s) \right)+\frac{2p-2}{r}z'_{0}(r)<0,r=\frac 1s,$$
which yields a contradiction.

\end{proof}

\begin{lemma}\label{superhar}
$u^*_i>0,\quad\text{in}\quad \mathbb R^{n+1}\backslash\{0\}$.
\end{lemma}
\begin{proof}
We need only to verify $u^*_{p-1}>0$, the proofs for the other $u^*_i$ are the same. If not, we can find $x_0\in\mathbb R^{n+1}\backslash\{0\}$ such that $u^*_{p-1}(x_0)=-c_0<0$. As the same arguments in the proof of Theorem \ref{mainthm2} in Section 3, we may assume that the $n+1$-th coordinate of $x_0$ is $0$. Set $v_i(x)=|x_{n+1}|^{2a-1}u^*_i(x)$. It is easy to see that $$\lim_{r\rightarrow 0}\frac1{r^{n+2a-1}}\int_{\partial B_r(x_0)}v_{p-1}(x)dS<0.$$
 Following the computation in Section 3 yields that $$\frac{d}{dr}\left[\frac1{r^{n+2a-1}}\int_{\partial B_r}v_{p-1}(x)dS\right]=\frac1{r^{n+2a-1}}\int_{\partial B_r}\left[\frac{\partial v_{p-1}}{\partial r}-(2a-1)r^{-1}v_{p-1}\right]dS<0.$$ Set $w(r,x_0)=\frac1{r^{n+2a-1}}\int_{\partial B_r(x_0)}v_{p-1}(x)dS$. Then we can see $w(r,x_0)<0,$ for $r\in [0,|x_0|)$. Therefore, by
Lemma \ref{div} and Theorem \ref{mainthm3}, we can choose $\epsilon$ small enough such that $$\int_{B_\rho(x_0)} v_{p-1}(x)dx<0,\rho=|x_0|+\epsilon.$$ Noting that we have $u^*_{p-1}>0$ near infinity, then we have for any $ x\in \mathbb R^{n+1}$
$$\frac1{r^{n+2a-1}}\int_{\partial B_r(x)}v_{p-1}dS>0,\text{ for }r>|x|.$$
Now we have \begin{eqnarray*}
\int_{B_\rho(x_0)}v_{p-1}(x)dx=\int_{B_{\epsilon}(0)}v_{p-1}dx+\int_{0}^1dt
\int_{\partial B_{t|x_0|+\epsilon}(tx_0)}v_{p-1}(x)dx>0.
\end{eqnarray*}This yields a contradiction.
\end{proof}

Now we are in a position to prove Theorem \ref{mainthm1}.
\par {\bf The proof of Theorem \ref{mainthm1}}:
\par Set $w_{\lambda}(x)=u^*(x)-u^*(x^\lambda)$ in $\Sigma_\lambda$.
From Lemma \ref{superhar} we get $u^*_i>0$. So by Lemma \ref{move1}  and Lemma \ref{lem005}, one gets
$u^*_{i}(x)\rightarrow 0$ as $x\rightarrow \infty$ to get
$$(-\widetilde\Delta_{n+1,a})^{p-i}w_{\lambda}>0,i=1,\cdots,p.$$ for all $\lambda\geq \lambda'>>1$. Set $$
\lambda_0=\inf\{\lambda>0\big|(-\widetilde\Delta_{n+1,a})^{p-i}w_{\mu}(x)>0\text{ in
}\Sigma_{\mu}\text{ for }\mu\geq \lambda,i=1,\cdots,p\}.$$
We may assume $\lambda_0>0$ and by the definition of $\lambda$, we see that \begin{equation*}
(-\widetilde\Delta_{n+1,a})^i w_{\lambda_0}(x)\geq 0,\quad i=0,\cdots,p-1,\end{equation*}
By Lemma \ref{prelem2}, we have either $$(-\widetilde\Delta_{n+1,a})^i w_{\lambda_0}(x)= 0,\quad i=0,\cdots,p-1.$$ Or $$(-\widetilde\Delta_{n+1,a})^i w_{\lambda_0}(x)> 0,\quad i=0,\cdots,p-1.$$ We need to prove the first situation is true. If not, there exist $\lambda_n\uparrow \lambda_0$ and $i_0\in\{1,\cdots,p-1\}$ such that $$\inf_{x\in \mathbb R^2}(-\widetilde\Delta_{n+1,a})^{i_0} w_{\lambda_n}(x)=\inf_{x\in B_R\backslash B_r}(-\widetilde\Delta_{n+1,a})^{i_0} w_{\lambda_n}(x)=(-\widetilde\Delta_{n+1,a})^{i_0} w_{\lambda_n}(x^{\lambda_n})<0$$ for some fixed $R$ large and $r$ small. This is a  conclusion from  Lemma \ref{move2}, Lemma \ref{superhar} and Lemma \ref{prelem2}.
There are two cases we should distinguish with:
\begin{itemize}
\item[(1)] $\displaystyle\lim_{n\rightarrow\infty} x^{\lambda_n}=x^0\in \Sigma_{\lambda}$. Then $$(-\widetilde\Delta_{n+1,a})^{i_0}w_{\lambda_0}(x^0)=\lim_{n\rightarrow\infty}(-\widetilde\Delta_{n+1,a})^{i_0}
    w_{\lambda_n}(x^{\lambda_n})\leq 0$$ which is a contradiction to $(-\widetilde\Delta_{n+1,a})^{i_0}w_{\lambda_0}(x)>0$.
\item[(2)] $\displaystyle\lim_{n\rightarrow\infty} x^{\lambda_n}=x^0\in \partial\Sigma_{\lambda}$. Then $$\partial_{x_1}(-\widetilde\Delta_{n+1,a})^{i_0}w_{\lambda_0}(x^0)=\lim_{n\rightarrow\infty}\partial_{x_1}
    (-\widetilde\Delta_{n+1,a})^{i_0}
    w_{\lambda_n}(x^{\lambda_n})= 0$$
    which is a contradiction to $\partial_{x_1}(-\widetilde\Delta_{n+1,a})^{i_0}w_{\lambda_0}(x^0)<0$.
\end{itemize}
If $\alpha<\frac{n+2a+2p}{n+2a-2p}$, we have $\tau>0$. To prove the radial symmetry of $u^*(x)$, one should take a
transformation. Set
$$\tilde{u}^*(x',x_{n+1},x_{n+2})=u^*(x',\sqrt{x_{n+1}^2+x_{n+2}^2}).$$
It follows that, \begin{equation}\label{301}
(-\Delta_{n+2,a-\frac 12})^p\tilde u^*=
|x|^{-\tau}(\tilde u^*)^\alpha,
\text{ in }\mathbb R^{n+2},\partial_{n+2}\tilde u^*(x',x_{n+1},0)=0,
\end{equation} here $\displaystyle \widetilde \Delta_{n+2,a-\frac 12}=\sum_{i=1}^{n+2}\partial_{x_i}^2+\frac{2a-2}{x_{n+2}}\partial_{x_{n+2}}$.
There is a
singularity at $0$, and hence $\lambda_0$ must be $0$. Notice that
 \eqref{301}  is rotationally invariant about $x', x_{n+1}$. For $|x'|^2+x_{n+1}^2=|\bar{x}'|^2+\bar{x}_{n+1}^2$, we have
$$u^*(x',x_{n+1})=\tilde u^*(x',x_{n+1},0)=\tilde  u^*(\bar{x}',\bar x_{n+1},0)=u^*(\bar x',\bar x_{n+1}).$$ This implies that $$\tilde {u}(x',x_{n+1})=\tilde u(\bar{x}',\bar x_{n+1}),\text{ if }|x'|^2+x_{n+1}^2=|\bar{x}'|^2+\bar{x}_{n+1}^2.$$
 If we take another transformation such
as
$$u^*_b(x)=\frac 1{|x|^{n+2a-2}}\tilde {u}_b\left(\frac{x}{|x|^2}\right),\text{ here
}b_{n+1}=0,$$  where $\tilde {u}_b(x)=\tilde {u}(x-b)$. Repeating the above arguments, similarly we have
$$\tilde u(x',x_{n+1})=\tilde u(\bar{x}',\bar x_{n+1}), \text{ if }|x'+b'|^2+x_{n+1}^2=|\bar{x}'+b'|^2+\bar x_{n+1}^2.$$ In fact, $b'$ can
be chosen arbitrarily, thus $\tilde u$ must be a constant. This means that $\tilde u\equiv 0$.

Now we consider the case $\alpha=\frac{n+2a+2}{n+2a-2}$ or $\tau=0$.
By the same arguments as we did in the case $\tau>0$,
there exists $\lambda=(\lambda_1,...,\lambda_{n+1})$ such that
\begin{equation}\label{109}
\tilde u^*(x',x_{n+1},0)=u^*(x',x_{n+1})=u^*(\bar{x}',\bar{x}_{n+1})
=\tilde u^*(\bar{x}',\bar{x}_{n+1},0),
\end{equation} if
 $\displaystyle \sum_{i=1}^{n+1}|x_i-\lambda_i|^2=
\sum_{i=1}^{n+1}|\bar{x}'_i-\lambda_i|^2$. In fact, $\lambda_{n+1}$ must be 0. Otherwise, it follows that $$
u^*(x',2\lambda_{n+1}-x_{n+1})=u^*(x',x_{n+1})=u^*(x',-x_{n+1}).$$ It shows
that for the fixed $x'$, $u^*$ is periodic with respect to $x_{n+1}$
with period $2\lambda_{n+1}$. This means that $u^*$ must vanish which
is impossible. For $\lambda'=(\lambda_1,...,\lambda_n)$, we have two
cases.\begin{itemize}
\item[(1)] $\lambda'=0$:  since
$\tilde u (x)=\frac{1}{|x|^{n+2a-2p}}u^*(\frac{x}{|x|^2})$, $\tilde{u}(x)$
is radially symmetric with respect to the origin.
\item[(2)] $\lambda'\neq 0$: This means that $0$ is not the
symmetric center of $u^*$, $u^*$ must be $C^2$ at $0$. In other words,
$\tilde{u}(x)$ has  similar asymptotic behaviors at $\infty$ as
$u^*(x)$. This allows us to apply the  moving plane method to
$\tilde {u}(x)$ directly to obtain that $\tilde {u}(x)$ is radially
symmetric with respect to some point $b\in \mathbb R^{n+1},b_{n+1}=0$.
\end{itemize}
The above arguments show  that $\tilde u(x)$ is radially symmetric with
respect to some point $b\in \{b_{n+1}=0\}$. Now we can follow the
arguments of Section 3 in \cite{ChenLiOu} and use the conformal invariant property to classify the solutions. This completes the proof of Theorem
\ref{mainthm1}.

\section{Appendix}
In the Appendix, we will prove \eqref{pre4}. We borrow the ideas from \cite{AxlerBourdonWade} and \cite{Pavlovic}. As polynomials are dense in $C^{2k}(B_1)$ and the operator $\widetilde \Delta_{n+1,a}$ is linear and local, we only need to show \eqref{pre4} is true for all homogeneous polynomials which are even functions with respect to $x_{n+1}$. At first, we need a decomposition for polynomials in $\mathbb R^{n+1}$. Denote the set of  all the polynomials which are even with respect to $x_{n+1}$ by $\mathscr{\widetilde{P}}$. Set
\begin{equation*}
\begin{split}
&\mathscr{\widetilde{P}}_m=\{p\in\mathscr{\widetilde{P}}|p \text{  is a homogeneous polynomials with order } m\},\\
&\mathscr{\widetilde{H}}_m=\{p\in\mathscr{\widetilde{P}}_m|\widetilde\Delta_{n+1,a} p=0\}.
\end{split}
\end{equation*}
\begin{lemma}\label{lemapp1}
If $p\in \mathscr{\widetilde{P}}$ with $\deg(p)=m$, then there exists some polynomial $q\in\mathscr{\widetilde{P}}$ with $\deg(q)\leq m-2$ satisfying that
$$\widetilde\Delta_{n+1,a}((1-|x|^2)q+p)=0.$$
\end{lemma}
\begin{proof}
We now define an operator $T: W\rightarrow W$  by
$$T(q)=\widetilde\Delta_{n+1,a}((1-|x|^2)q),\quad q\in W$$
where $W$ is the set of all polynomials belong to $\mathscr{\widetilde{P}}$ with order less or equal to $m-2$. First we show $T$ is injective. If $T(q)=0$, this means that $(1-|x|^2)q$ solves
\begin{equation*}
\widetilde\Delta_{n+1,a}((1-|x|^2)q)=0,\text{ in }B_1, \quad (1-|x|^2)q=0,\text{ on }\partial B_1.
\end{equation*}
By  Lemma \ref{prelem2}, we must have $q=0$ which implies $T$ is injective. Note that $W$ is a finite dimension vector space. This means $T$ is also surjective. Hence we have for any $p\in \mathscr{\widetilde{P}}$ with $\deg(p)=m$, there exists $q\in W$ such that
$$\widetilde\Delta_{n+1,a}((1-|x|^2)q)=-\widetilde\Delta_{n+1,a}p.$$
\end{proof}
By Lemma \ref{lemapp1}, we have for any $p\in \mathscr{\widetilde{P}}_m$ there exists $q\in \mathscr{\widetilde{P}}$ with $\deg(q)\leq m-2$ such that
\begin{equation}\label{app1}
p=h+|x|^2q-q
\end{equation}where $h\in \mathscr{\widetilde{P}}$ and $\widetilde\Delta_{n+1,a} h=0$. Also from the decomposition we know $\deg(h)\leq m$. Taking the homogeneous part of degree $m$ at both sides, we get
$$p=p_m+|x|^2q_{m-2},p_m\in \mathscr{\widetilde{H}}_m, q_{m-2}\in \mathscr{\widetilde{P}}_{m-2}.$$
Repeating the above decomposition, we get for $p \in \mathscr{\widetilde{P}}_m$,
\begin{equation*}
p=p_m+|x|^2p_{m-2}+|x|^4p_{m-4}+\cdots, p_{j}\in \mathscr{\widetilde{H}}_j.
\end{equation*}
The summation of the above decomposition is finite. By such a decomposition, we only need to show \eqref{pre4} is true for $u(x)=|x|^{t-k} h(x)=|x|^t h(\frac{x}{|x|})$ where $t\in R$ and $h\in \mathscr{\widetilde{H}}_k$ for some $k$.
\begin{equation*}
\widetilde\Delta_{n+1,a} u(x)=[t(t+n+2a-2)-k(k+n+2a-2)]|x|^{t-2} h\left(\frac{x}{|x|}\right)
\end{equation*}if we note $ x\cdot\nabla h=kh$. Applying the above identity for $m$ times, we get
$$(\widetilde\Delta_{n+1,a})^{m} u(x)=A_{m,t}|x|^{t-2m}h\left(\frac{x}{|x|}\right)$$
where $$A_{m,t}=\prod_{j=0}^{m-1}[(t-2j)(t-2j+n+2a-2)-k(k+n+2a-2)].$$
As for $|x|^{2m-n-2a-t}h(\frac{x}{|x|})$, we have
$$(\widetilde\Delta_{n+1,a})^m \left(|x|^{2m-n-2a-t}h\left(\frac{x}{|x|}\right)\right)=B_{m,t}|x|^{-n-2a-t}h\left(\frac{x}{|x|}\right)$$ where
$$B_{m,t}=\prod_{j=0}^{m-1}[(2m-n-2a-t-2j)(2m-t-2j-2)-k(k+n+2a-2)].$$
It is easy to see that $A_{m,t}=B_{m,t}$ and this proves \eqref{pre4}.

\end{document}